\documentclass[a4wide,11pt]{amsart}

\usepackage{amsmath,url}
\usepackage{amssymb}
\usepackage{color}
\usepackage{graphicx}
\usepackage{bbm}
\usepackage{a4wide}
\usepackage{t1enc}
\usepackage{amsthm}
\newtheorem{theorem}{Theorem}[section]
\newtheorem{lemma}[theorem]{Lemma}
\newtheorem{cor}[theorem]{Corollary}

\theoremstyle{definition}
\newtheorem{definition}[theorem]{Definition}
\newtheorem{example}[theorem]{Example}
\newtheorem{remark}[theorem]{Remark}

\numberwithin{theorem}{section}

\newcommand{\sos}{\operatorname{SOS}}
\newcommand{\soc}{\operatorname{SOC}}

\newcommand{\sosg}{\operatorname{SOS(\varGamma)}}
\newcommand{\socg}{\operatorname{SOC(\varGamma)}}

\newcommand{\so}{\operatorname{SO}}

\newcommand{\g}{\varGamma}
\newcommand{\OO}{\mathcal O}
\newcommand{\sca}{\operatorname{scal}_\varGamma}

\newcommand{\re}{\right}
\newcommand{\li}{\left}

\newcommand{\NN}{\mathbbm{N}}
\newcommand{\CC}{\mathbbm{C}}
\newcommand{\QQ}{\mathbbm{Q}}
\newcommand{\RR}{\mathbbm{R}}
\newcommand{\ZZ}{\mathbbm{Z}}
\newcommand{\s}{\mathbb S}
\newcommand{\ovln}{\overline}
\newcommand{\klg}{\leqslant}

\DeclareMathOperator{\Ker}{Ker}

\newcommand{\bmidb}[2]{\{\nonscript\,{#1}\mid{#2}\nonscript\,\}}
\newcommand{\Bmidb}[2]{\left\{\,{#1}\mathrel{\hbox{$
  \displaystyle \mathsurround=0pt \nulldelimiterspace=0pt
  \left|\vphantom{{#1}{#2}}\right.$}}{#2}\,\right\}}

\newdimen{\standardlabelwidth}
\standardlabelwidth-\standardlabelwidth
  \advance\standardlabelwidth by 2\normalparindent
\newcommand{\standardlabel}[1]{#1\kern\standardlabelwidth}
\begin{document}
\title{Similarity versus Coincidence Rotations of Lattices}
%    author one information
\author{Svenja Glied}
\address{Fakult\"at f\"ur Mathematik, Universit\"at Bielefeld, Postfach 100131, 33501 Bielfeld, 
       Germany}
\curraddr{}
\email{\{sglied,mbaake\}@math.uni-bielefeld.de}
\urladdr{http://www.math.uni-bielefeld.de/baake/ }

\thanks{}

%    author two information
\author{Michael Baake}
\address{}
\curraddr{}
\email{}
\thanks{}

\begin{abstract}
           The groups of similarity and coincidence rotations of an arbitrary lattice $\g$ in 
          $d$-dimensional Euclidean space are considered. It is shown that the 
          group of similarity rotations contains the coincidence rotations as a normal subgroup.
          Furthermore, the structure of the corresponding factor group is examined.
          If the dimension $d$ is a prime number, this factor group is an elementary Abelian
          $d$-group. Moreover, if $\g$ is a rational lattice, the factor group is 
          trivial ($d$ odd) or an elementary Abelian $2$-group ($d$ even).
\end{abstract}
\maketitle
\section{Introduction}\label{intro}
The classification of colour symmetries and that of grain boundaries in crystals 
and quasicrystals are
intimately related to the existence of similar and coincidence
sublattices of the underlying lattice of periods or the corresponding
translation module. It is thus of interest to understand the
corresponding groups of isometries from a more mathematical perspective.
An example for the structure of the groups of coincidence rotations
and similarity rotations of planar lattices is considered and the
factor group of similarity modulo coincidence rotations is calculated.
More generally, for lattices in $d$ dimensions, we show that the factor group is
the direct sum of cyclic groups of prime power orders that divide $d$. In
the case of rational lattices, which include hypercubic
lattices and all root lattices, this means that the factor group is
either trivial or an elementary Abelian $2$-group, depending on the parity
of $d$.
\section{Coincidence Rotations}
A \emph{lattice} in $\RR^d$ is a subgroup of the form
$$\g = \ZZ b_1 \oplus \ZZ b_2 \oplus \ldots \oplus \ZZ b_d,$$
where $\{b_1,\ldots,b_d\}$ is a basis of $\RR^d$.  
Two lattices $\g,\g'$ in $\RR^d$ are called 
\emph{commensurate} if their 
intersection $\g \cap \g'$ has finite index both in $\g$ and in $\g'$.
In this case, we write $\g \sim \g'$.
Commensurateness of lattices is an equivalence relation (cf.~\cite{Baake}).
An element $R \in \so(d)$ is called a 
\emph{coincidence rotation of $\g$}, if $\g \sim R\g$. We thus define 
$$ \socg:= \bmidb{R \in \so(d)}{\g \sim R\g},$$
which is a subgroup of $\so(d)$.
\begin{example}[The square lattice $\ZZ^2$]
 As shown in Thm.~3.1 of \cite{Baake}, 
 the coincidence rotations of $\ZZ^2$ are precisely the special 
 orthogonal matrices with rational entries, 
 \begin{equation*}
   \soc(\ZZ^2) = \so(2,\QQ).
 \end{equation*}
 On the other hand, one can identify $\ZZ^2$ with the Gaussian integers $\ZZ[i]$, where $i$ is
 the imaginary unit. 
 Then, a rotation $R(\varphi)$ with rotation angle $\varphi$ corresponds 
 to a multiplication with the complex number 
 $e^{i\varphi} \in \li(\QQ(i)\cap \s^1\re)\simeq \soc(\ZZ^2)$; see \cite{Roth}.
 Using the fact that $\ZZ[i]$ is a unique factorisation domain, each coincidence rotation 
 uniquely factorises as 
 \begin{equation}\label{Zerlegung}
  e^{i\varphi} = 
         \varepsilon\prod_{p\equiv 1(4)}\li(\frac{\,\omega_p\,}{\ovln{\omega_p}}\re)^{n_p},
 \end{equation}
 where $\varepsilon$ is a unit in $\ZZ[i]$, $n_p \in \ZZ$ with only finitely many of them 
 nonzero, $p$ runs through the rational primes congruent to $1$ (mod $4$), and $p$ factorises
 as $p = \omega_p\ovln{\omega_p} $ in $\ZZ[i]$ with $\omega_p/\ovln{\omega_p}$ not a 
 unit.
 This shows that $\soc(\ZZ^2)$ is a countably generated Abelian group. More precisely, 
 \begin{equation*}
   \soc(\ZZ^2) = C_4 \times \ZZ^{(\aleph_0)},
 \end{equation*}
 where $C_4$ denotes the cyclic group of order $4$ (here generated by $i$) and 
 $\ZZ^{(\aleph_0)}$ stands for
 the direct sum of countably many infinite cyclic groups, which are here generated by the 
 $\omega_p/\ovln{\omega_p}$ with $p\equiv 1$ (mod $4$) (cf.~\cite{Roth}).
 
\end{example}

%irgendwo einbauen: The group $\soc(\ZZ^2)$ is Abelian, hence $\sos(\ZZ^2)$ is a normal 
% subgroup.
\section{Similarity Rotations}

Let $\g \subset \RR^d$ again be a lattice. Define
\begin{displaymath}
   \sosg := \bmidb{R \in \so(d)}{\g \sim \alpha R \g 
               \text{\,\, for some\,\,} \alpha > 0}.
\end{displaymath}
The elements of $\sosg$ are called \emph{similarity rotations}.
$\sosg$ is a group (cf.~\cite{Heu}) and contains $\socg$ as a subgroup.

\begin{example}[$\ZZ^2$]

 For $\ZZ^2$, the group of similarity rotations consists precisely of the set of 
 $\ZZ^2$-directions, 
 \begin{equation}\label{sos Richtungen}
   \sos(\ZZ^2) = \Big\{\frac{a}{|a|} \,\Big | \, 0\not = a \in \ZZ[i]\Big\}. 
               %\Bmidb{\frac{a}{|a|}}{0\not = a \in \ZZ[i]}.
 \end{equation}
  We parametrise the Euclidean plane by the complex numbers $\CC$, and 
  use $\so(2)\simeq \s^1$ and $\ZZ^2=\ZZ[i]$.
  To show \eqref{sos Richtungen}, let $z \in \ZZ[i]\setminus \{0\}$. 
  Since $\ZZ[i]$ is a ring, one has
  $$ |z|\cdot \frac{z}{|z|}\,\ZZ[i]\subset \ZZ[i],$$
  so that $z/|z| \in \sos(\ZZ^2)$.
  Conversely, let $r \in \sos(\ZZ^2)$, meaning that 
  $r \in \s^1$  with $\lambda r \ZZ[i] \sim \ZZ[i]$ for 
  some $\lambda > 0$. By Remark~\ref{den} below, there exists a nonzero integer $t$ with 
  $t\lambda r\ZZ[i] \subset \ZZ[i]$.
  Since $1 \in \ZZ[i]$, this yields
  $t\lambda r \in \ZZ[i]$, say $t\lambda r = v$.
  Thus $|t\lambda|=|v|$, because $r \in \s^1$.
  This shows that $r =  v/|v|$ is a $\ZZ[i]$-direction.

 Each nonzero element of $\sos(\ZZ^2)$ is thus of the form $z/|z|$ with $0\not=z \in \ZZ[i]$.
 Using unique factorisation in $\ZZ[i]$ again, we get
 \begin{equation*}
   \frac{z}{|z|} = \li(\frac{1+i}{\sqrt 2}\re)^k \prod_{p\equiv 1(4)}
                   \li(\frac{\omega_p}{\sqrt p}\re)^{\ell_p},
 \end{equation*}
 where $0\klg k < 8$ and $\ell_p \in \ZZ$ (other restrictions as in~\eqref{Zerlegung}).
 One observes that $(1+i)/\sqrt 2$ is a primitive $8$th root of unity, hence it 
 generates the cyclic group $C_8$.
 Furthermore, one finds
 %\begin{equation*} 
 % $ \li(\frac{1+i}{\sqrt 2}\re)^2 = i 
 %\end{equation*}
 %and
 \begin{equation*} 
   \li(\frac{\omega_p}{{\sqrt p }}\re)^2 = \frac{\omega_p^2}{\omega_p \ovln{\omega_p}} = 
                                               \frac{\,\omega_p\,}{\ovln{\omega_p}}.
 \end{equation*}
 This shows that the generators of $\soc(\ZZ^2) = C_4 \times \ZZ^{(\aleph_0)}$ are the squares
 of the generators of $\sos(\ZZ^2)$.
 Thus 
 \begin{equation*}
   \soc(\ZZ^2) = \Bmidb{x^2}{x \in \sos(\ZZ^2)}
               =: \li(\sos(\ZZ^2)\re)^2.
 \end{equation*}
  The following more general result was shown in~\cite{Roth}: For all cyclotomic fields 
 $\QQ(\xi_n)$
 of class number one (excluding $\QQ$), one has
 \begin{equation*}
   \soc(\OO_n) \simeq C_{N(n)} \times  \ZZ^{(\aleph_0)},
 \end{equation*} 
 where $\OO_n=\ZZ[\xi_n]$ is the ring of integers in $\QQ(\xi_n)$ and 
 $N(n)= \operatorname{lcm}(n,2)$. 

 Returning to our example, we find the structure of the factor group to be
 \begin{eqnarray*}
   \sos(\ZZ^2)/\soc(\ZZ^2)  &\simeq &  \li(C_8/C_4\re) \times C_2^{(\aleph_0)}\\
                           &\simeq & C_2 \times  C_2^{(\aleph_0)},  
 \end{eqnarray*}
 where $ C_2^{(\aleph_0)}$ stands for the direct sum of countably many cyclic groups of order $2$.
 Hence, the factor group is the direct sum of cyclic groups of order $2$, which means that it is an 
 elementary Abelian $2$-group.              %Referenz angeben!!!!!
 More generally, for arbitrary lattices in Euclidean $d$-space, we shall see below that the
 group $\soc$ is a normal subgroup of $\sos$, whence the factor group always exists.
\end{example}

\section{Factor Group}

Throughout this section, let $\varGamma$ be a lattice in $\RR^d$, with $d \geq 2$.
\begin{definition}
  For an arbitrary element $R \in \so(d)$, define
  \begin{equation*}
     \sca(R)= \bmidb{\alpha \in \RR}{\varGamma \sim \alpha R \varGamma}.
  \end{equation*}
\end{definition}
Note that $$\sos(\g) = \bmidb{R \in \so(d)}{\sca(R)\not = \varnothing}.$$ 
\begin{remark}\label{den}
  If $\alpha \in \sca(R)$, then there exists a nonzero integer $t$ such that 
  $t\alpha R\g \subset \g$. Namely, if $\alpha~\in~\sca(R)$, the group index 
  $[\alpha R\g :(\g \cap \alpha R\g)]=t$ is finite. 
  Consequently, one has $t\alpha R\g \subset (\g\cap \alpha R\g)\subset \g$.
\end{remark}

\begin{lemma}\label{scal}
 For $R \in \sos(\varGamma)$, the following 
 assertions hold. 
  \begin{enumerate}
    \item[{\rm (1)}]    
       $ b\cdot \sca(R) \subset \sca(R)$ \/ for all \/ $b \in \QQ\setminus\{0\}$     
    \item[{\rm (2)}]
       $ r\varGamma \sim \varGamma$ \/ with \/ $r\in \RR$ \/ implies \/ $r \in \QQ$
    \item[{\rm (3)}] 
       $ \alpha\beta^{-1} \in \QQ$ \/ for all \/ $\alpha, \beta \in \sca(R)$       
  \end{enumerate}
\end{lemma}
\begin{proof}
 Let $\alpha \in \sca(R)$.
 For $b=b_1/b_2$ with $b_1,b_2 \in \ZZ\setminus\{0\}$, one finds
 \begin{equation*}
    \frac{b_1}{b_2}\alpha R \varGamma \sim \frac{1}{b_2} \alpha R \varGamma 
         \sim \frac{1}{b_2} \varGamma \sim \varGamma.
 \end{equation*}
  This proves~(1).  
  In order to show~(2), let $r\in \RR$ with $r\g\sim\g$. By Remark~\ref{den},
  there exists a nonzero integer $k$ with $kr\g \subset \g$.
  Now, let $\gamma \in \g$ be represented in terms of a basis 
  $\{\gamma_1,\ldots,\gamma_d\}$ of $\g$ as 
  $\gamma = \sum_{i=1}^d c_i\gamma_i$, with $c_i \in \ZZ$.
  On the other hand, $kr\gamma$ can be represented as 
  $kr\gamma = \sum_{i=1}^d a_i\gamma_i$, where $a_i \in \ZZ$. Thus
  $$\sum_{i=1}^d krc_i\gamma_i = \sum_{i=1}^d a_i\gamma_i.$$
  By assumption, $\g$ spans $\RR^d$, so that $\{\gamma_1,\ldots,\gamma_d\}$ forms an 
  $\RR$-basis of 
  $\RR^d$.
  Therefore, one has $krc_i=a_i$, yielding $r=a^{}_ic_i^{-1}k^{-1} \in \QQ$.
  Finally, (3) is obtained from~(2) as follows. By assumption, 
  one has
  \begin{equation*}
     \beta R\varGamma \sim \varGamma \sim \alpha R \varGamma.
  \end{equation*}
  Multiplying with $1/\beta$ gives
      $ R \varGamma \sim \frac{\alpha}{\beta} R \varGamma$,
  which completes the proof.  
 \end{proof}
  
Denote by $\RR^\bullet$ (by $\QQ^\bullet$) the multiplicative groups formed by the 
nonzero real (rational) numbers. Define \label{eta} a map
\begin{equation*} 
   \eta\colon \sos(\varGamma) \longrightarrow \RR^\bullet/\QQ^\bullet\\
\end{equation*}
by
\begin{equation*}
   \qquad \,\, R \longmapsto [\alpha],
\end{equation*}
where $[\,\cdot \,]$ denotes the equivalence classes of $\RR^\bullet/\QQ^\bullet$ and $\alpha $ 
is an arbitrary element of $\sca(R)$.
This map is well-defined due to the fact that $\sca(R)$ is non-empty for $R \in \sos(\varGamma)$
and by Lemma~\ref{scal}(3). 
\begin{lemma}\label{eta hom}
  The map $\eta$ is a group homomorphism with $\Ker(\eta)=\soc(\varGamma)$.
\end{lemma}
\begin{proof}
 Let $R,S \in \sos(\varGamma)$ and choose $\alpha \in \sca(R)$ and
 $\beta \in \sca(S)$. We need to show that $\alpha\beta \in \sca(RS)$.
 By assumption, one has
 \begin{equation*}
   \varGamma \sim \alpha R \varGamma \sim \alpha R (\beta S \varGamma) = \alpha\beta RS\varGamma.
 \end{equation*}
  Thus $\alpha\beta \in \sca(RS)$, hence $\eta$ is a group homomorphism.
It remains to show that $\Ker(\eta)=\soc(\varGamma)$.
For $R \in \soc(\varGamma)$, the set $\sca(R)$ contains $1$, which means 
$R~\in ~\Ker(\eta)$.
Conversely, if $S \in \Ker(\eta)$, one has $\sca(S) \subset \QQ$. Let $\mu \in \sca(S)$. 
Due to Lemma~\ref{scal}(1),
we have $1 =\mu^{-1}\mu \in \sca(S)$, which proves $S~\in~\soc(\varGamma)$.
\end{proof}

Since $\soc(\varGamma)$ is the kernel of a group homomorphism, it is a normal subgroup of 
$\sos(\varGamma)$, so that the factor group $\sosg/\socg$ can be considered. 
It is isomorphic to the image of
$\eta$, which is a subgroup of $\RR^\bullet/\QQ^\bullet$ and thus Abelian.
To examine the structure of the factor group $\sosg/\socg$, we need the following result 
from the theory of Abelian groups.
\begin{theorem}\label{direct sum} Let G be a countable Abelian group.
 \begin{enumerate}
 \item[{\rm (1)}]
   If a prime number $p$ exists such that $x^p=1$ for all $x \in G$, 
   then $G$ is the direct 
   sum of subgroups of order $p$.  
 \item[{\rm (2)}]
   If a positive integer $n$ exists such that $x^n=1$ for all \,$x \in G$, then 
   G is the direct sum of cyclic groups of prime power orders that divide $n$. 
 \end{enumerate}
 
\end{theorem}
\begin{proof}
 See \cite[Thms.\ 5.1.9 and 5.1.12]{Scott}.%andere Referenz suchen!!!!!!!!!! abzaehlbare Version...
\end{proof}

\begin{remark} \label{delta} Let $R \in \sosg$.
  For all elements $\alpha\in \RR$ with $\alpha R\g\subset \g$, one has 
  $|\alpha^d|=[\g : \alpha R \g] \in \NN$. This follows via the determinants of  
  basis matrices of the lattices involved. 
  Consequently, $\alpha$ is an algebraic number. 
\end{remark}

\begin{theorem}
  The group $\sosg/\socg$ is countable. Furthermore, it is the direct sum of cyclic groups of 
  prime power orders that divide $d$.
\end{theorem}
\begin{proof}
  We consider again the group homomorphism $\eta\colon \sosg \longrightarrow 
  \RR^\bullet/\QQ^\bullet$.  Let $R \in \sosg$.
  This implies $\eta(R)= [\alpha]$ for some element $\alpha \in\sca(R)$. 
  Due to Remark~\ref{den}, there exists a nonzero integer $t$ with $t\alpha R\g \subset \g$.
  Furthermore, one has
  $\eta(R) = [t\alpha]$.
  By Remark~\ref{delta}, $t\alpha$ is an algebraic number. This means that all 
  elements 
  of $\eta(\sosg)$ are represented by algebraic numbers. 
  Thus, since the set of algebraic numbers is countable, also the group
  $\sosg/\socg$ is countable.

  According to Remark~\ref{delta}, one has $(t\alpha)^d \in \QQ$, which yields  
  \begin{equation}\label{delta^d is 1}
    \eta(R)^d=[t\alpha]^d = [(t\alpha)^d] = [1]
  \end{equation}
  in $\RR^\bullet/\QQ^\bullet$.
  Using the group isomorphism $\eta(\sosg) \simeq \sosg/\socg$, this shows that the order 
  of each element of $\sosg/\socg$ divides $d$.
  Theorem~\ref{direct sum}(2) then implies that 
  the group 
  $\sosg/\socg$ is the direct sum of cyclic groups of prime power orders. 
  %On the other hand, 
  %the order of every element of $\sosg/\socg$ is a divisor of $d$ by 
  %equation~\eqref{delta^d is 1}.
  Consequently, the prime power order of each cyclic group divides $d$.
\end{proof}
\begin{cor}
  If $d=p$ is a prime number, the factor group $\sosg/\socg$ is an elementary
  Abelian $p$-group, i.e., it is the direct sum of cyclic groups of order $p$. $\hfill\square$
    
\end{cor}
%The \emph{dual lattice} $\g^*$ of $\g$ is defined by  
%$$\g^*= \bmidb{x \in \RR^d}{\langle x,y\rangle \in \ZZ \text{\,\,for all\,\,} y \in \g}.$$

%\begin{theorem}
%    Let $\g$ satisfy one of the 
%    following conditions.
%    \begin{itemize}
%    \item $\g \sim \g^*$
%    \item $\g$ is a rational lattice, i.e., 
%          $\langle x, x\rangle \in \QQ$ for all $x \in \g$
%    \end{itemize}
%    Then $\sosg/\socg$ is an elementary Abelian $2$-group, if $d$ is even.
%    If $d$ is odd, one has $\sosg=\socg$.
%\end{theorem}
\begin{cor}[Rational Lattices] Let $\g$ be a lattice in $\RR^d$ such that 
  $\langle x , x \rangle \in \QQ$
  for all  $x \in \g$, where $\langle\, \cdot\,\, , \,\cdot\, \rangle$ denotes the standard scalar product 
  in $\RR^d$. Lattices satisfying the above property 
  are also called \,{\em rational} {\rm (cf.~\cite{Conway})}. 
  For these lattices, the group $\sosg/\socg$ is 
  an elementary Abelian $2$-group when $d$ is even.
  If $d$ is odd, one has $\sosg=\socg$. Either way, one has 
  \begin{equation*}
     \li(\sosg\re)^2 \subset \socg.
  \end{equation*} 
\end{cor}
\begin{proof}
  Let $R \in \sosg$. By Remark~\ref{den}, there exists a nonzero real number $\alpha$ 
  such that $\alpha R \g \subset \g$. 
  By assumption, one has
   $\langle \alpha R \gamma,\alpha R \gamma\rangle \in \QQ$\, for all $\gamma \in \g$. 
  Hence
  $\alpha^2\in \QQ$, say $\alpha^2=r/s$, where 
  $r, s \in\ZZ\setminus\{0\}$. 
  Since $s\alpha^2=r\in\ZZ$ and $\alpha R\g\subset\g$, 
  one gets
  $$
  \g\supset s\alpha R(\alpha R\g)=s\alpha^2R^2\g
  \subset R^2\g,
  $$
  whence
  $$
  rR^2\g\subset \li(\g \cap R^2\g\re).
  $$
  Thus both $[\g:rR^2\g]$ and $[R^2\g:rR^2\g]$ are finite. This implies
  $\g\sim R^2\g$, so that $R^2$ is a coincidence
  rotation of $\g$. Consequently,\, $\li(\sosg\re)^2\subset\socg$.
  This means that every element of the factor group $\sosg/\socg$ is of order $1$ or $2$. 
  Thus, the factor group is an elementary Abelian $2$-group by 
  Theorem~\ref{direct sum}(1).
 
  If $d$ is odd, set $d = 2m+1$ with $m \in \NN$. Then
   \begin{equation*}
     \alpha(\alpha^2)^m = \alpha^d \,\, \in \QQ
   \end{equation*}
   yields $\alpha \in \QQ$, because $\alpha^2 \in \QQ$. 
   Thus $\eta(R)=[\alpha] = [1]$ in $\RR^\bullet/\QQ^\bullet$ for all $R \in \sosg$, whence 
   $\sosg/\socg$ is the trivial group.
   % Therefore, one has $\sosg = \socg$ for $d$ odd.
\end{proof}

\section{Outlook}
In view of Penrose tilings and similar models,
where the translation module is not a lattice, 
it is desirable to generalise the above notions of 
similarity and coincidence rotations from lattices to modules.
Some progress has been made in this direction for certain modules over subrings $S$ of the 
rings of integers of real algebraic number fields. 
More precisely, similar results \cite{G} to those presented here hold for $S$-modules of 
rank $d$ that span $\RR^d$.

\section*{Acknowledgements}
The authors are grateful to U.\ Grimm, C.\ Huck, R.V.\ Moody and P.\ Zeiner 
for valuable discussions and 
comments on the manuscript. This work was supported by the German Research Council (DFG), 
within the CRC $701$. S.G.\ would like to thank the IUCr for financial support to attend 
ICQ10.


\begin{thebibliography}{9}
 
\bibitem{Baake} Baake, M.: Solution of the coincidence problem in
    dimensions $d\leq 4$. In: {\em The Mathematics
   of Long-Range Aperiodic Order} (Ed.\ R.V.\ Moody), p. 9--44. 
   NATO-ASI Series C {\bf 489},
   Kluwer,
   Dordrecht 1997; revised version: \texttt{arXiv:math/0605222 [math.MG]}.
\bibitem{Conway}
   Conway, J.H.; Rains, E.M.; Sloane, N.J.A.: On the existence of similar sublattices, 
   Canad. J. Math. {\bf 51} (1999) 1300--1306.
\bibitem{G}
   Glied, S.: Similarity and coincidence isometries for modules, in preparation.
\bibitem{Heu}
   Baake, M.; Grimm, U.; Heuer, M.; Zeiner, P.: Coincidence rotations of the root lattice 
   $A_4$,
   Europ. J. Combinatorics, in press. \texttt{arXiv:0709.1341 [math.MG]}.   
\bibitem{Roth}
   Pleasants, P.A.B.; Baake, M.; Roth, J.:  Planar coincidences for $N$-fold symmetry, 
    J.\ Math. Phys. {\bf 37} (1996) 1029--1058, revised version: 
   \texttt{arXiv:math/0511147 [math.MG]}.
\bibitem{Scott}
   Scott, W.R.: Group Theory, Prentice-Hall, Englewood Cliffs 1964.

\end{thebibliography}
\end{document}